\documentclass{amsart}
\usepackage[colorlinks=true, pdfstartview=FitV, linkcolor=blue, citecolor=blue, urlcolor=blue]{hyperref} 
\usepackage{lipsum}
\usepackage{graphicx}
\usepackage{amsmath,leftidx,amsthm}
\usepackage[all]{xy}
\usepackage{xspace}
\usepackage{amssymb}
\usepackage[latin1]{inputenc}
\usepackage{graphicx,color}
\usepackage{fullpage}
\usepackage{versions}
\usepackage{textcomp,listings}
\usepackage{enumitem}
\setenumerate[1]{label=(\arabic*)}

\newtheorem*{mthm}{Main Theorem}
\newtheorem{theorem}{Theorem}[section]
\newtheorem{lemma}[theorem]{Lemma}
\newtheorem{cor}[theorem]{Corollary}

\theoremstyle{definition}

\theoremstyle{remark}

\numberwithin{equation}{section}

    \newcommand{\ran}{\mbox{\rm ran\ }}
    
    \renewcommand{\phi}{\varphi}

 	\newfont{\caps}{cmcsc9}  
 	\newfont{\jour}{cmti9}  
\theoremstyle{remark}
\newtheorem{exmp}{\bf Example}

\definecolor{green}{rgb}{0,0.5,0}
\definecolor{dkgreen}{rgb}{0,0.6,0}
\definecolor{gray}{rgb}{0.5,0.5,0.5}
\definecolor{mauve}{rgb}{0.58,0,0.82}
\definecolor{purple}{rgb}{0.58,0,0.62}

\renewcommand{\phi}{\varphi}

\def\and{{\quad\text{and}\quad}}

\begin{document}

\title{Powers of posinormal Hilbert-space operators  }
\author{Paul S. Bourdon}
\address{Department of Mathematics\\  University of Virginia\\ Charlottesville, VA 22903}
\email{psb7p@virginia.edu}
\author{C.S.\ Kubrusly}
\address{Department of Mathematics\\Catholic University of Rio de Janeiro\\  22453-900, RJ, Brazil}
\email{carlos@ele.puc-rio.br}
\author{Derek Thompson}
\address{Department of Mathematics, Taylor University,  Upland, IN 46989}
\email{theycallmedt@gmail.com}

\begin{abstract}
A bounded linear operator $A$ on a Hilbert space $\mathcal{H}$ is \textit{posinormal} if there exists a positive operator $P$ such that $AA^{*} = A^{*}PA$.  We show that if $A$ is posinormal with closed range, then $A^n$ is posinormal and has closed range for all integers $n\ge 1$.    Because the collection of posinormal operators includes all hyponormal operators, we obtain as a corollary that powers of closed-range hyponormal operators continue to have closed range.   We also present a simple example of a closed-range operator $T: \mathcal{H}\to \mathcal{H}$ such that $T^2$ does not have closed range.  
\end{abstract}

\maketitle




\section{Introduction} 

A bounded linear operator $A$ on a Hilbert space $\mathcal{H}$ is said to be \textit{posinormal} if there exists a positive operator $P$ such that $AA^{*} = A^{*}PA$. The operator $P$ is called an \textit{interrupter} of $A$.  Note that if $A$ has interrupter $I$, then $A$ is normal.  Rhaly introduced the notion of posinormality in \cite{Rhaly}.  Posinormality is a unitary invariant; in fact, as Rhaly points out, if $V$ is an isometry (so that $V^*V = I$) and $A$ is posinormal with interrupter $P$, then $VAV^*$ is posinormal with interrupter $VPV^*$.   If $A^*$ is posinormal, then $A$ is \textit{coposinormal}.    

 Applying  a result  due to Douglas  \cite[Theorem 1]{douglas1966majorization}, Rhaly \cite[Theorem 2.1]{Rhaly} obtains useful equivalent conditions for posinormality: 
\begin{theorem}[\cite{Rhaly}]\label{posinormal}
For $A \in \mathcal{B(H)}$, the space of bounded linear operators on a Hilbert space $\mathcal{H}$, the following statements are equivalent:
\begin{enumerate}
    \item[(1)] $A$ is posinormal;
    \item[(2)]  $\ran \textrm{ } A \subseteq \ran \textrm{ } A^*$;
    \item[(3)]  $AA^{*} \leq \lambda^2 A^{*}A $ for some $\lambda \geq 0$; and
    \item[(4)]  there exists $T \in \mathcal{B(H)}$ such that $A = A^{*}T$.
\end{enumerate}
\end{theorem}
Note that condition (2) of the preceding theorem yields another way to see a normal operator $N$ is both posinormal and coposinormal because $\ran N = \ran N^*$.   By (3) with $\lambda = 1$, we see that every hyponormal operator is posinormal. 

  This paper draws its motivation from three sources \cite{Bouldin,BoT,Kubrusly}. 
  \begin{itemize}
  \item   In  \cite[Section 4]{BoT}, a class of  composition operators on the Hardy space $H^2$ of the open unit disk in the complex plane is identified such that (i) each member of the class is both posinormal and coposinormal  and (ii) for each positive integer $n$ exceeding $1$, there is a composition operator $C$ in this class for which $C^n$ is not posinormal.  
  \item In \cite{Kubrusly}, an example of a posinormal operator whose square is not posinormal is exhibited and a study of powers of posinormal operators is undertaken based on the notions of ascent and descent of Hilbert-space operators. Corollary 1(b)  of that paper states, contrary to the properties of the class of composition operators discussed in the preceding bullet item,  that if a Hilbert-space operator $T$ is both posinormal and coposinormal, then $T^n$ is posinormal for all $n\ge 1$.   An additional hypothesis is needed to make Corollary 1(b) of \cite{Kubrusly} hold. In this paper, we show that such a hypothesis is that the range of $T$ be closed:
\begin{mthm}
 If $T: \mathcal{H}\to \mathcal{H}$ is a posinormal operator with closed range, then $T^n$ is posinormal for all $n\ge 1$ (no assumption of coposinormality required), and $T^n$ has closed range for all $n\ge1$. 
\end{mthm}
\smallskip

\noindent {\sl Remark:  After posting v1 of this preprint at arxiv.org, the authors discovered that the preceding ``Main Theorem'' is Corollary 14 of \cite{JV}, which arises in a line of investigation distinct from that started by Rhaly in 1994 when he introduced the notion of posinormality.  Generalizing the notion of ``EP matrix'', Campbell and Meyer (\cite{CampMeyer}, 1975) introduced EP operators, which may be characterized as follows: a Hilbert space operator $T$ is EP provided that $T$ has closed range and $\ran\, T = \ran\, T^*$.   Itoh (\cite{Itoh}, 2005) introduced hypo-EP operators,  which may be characterized as follows: a Hilbert space operator $T$ is hypo-EP provided that $T$ has closed range and $\ran\, T \subseteq \ran\, T^*$.  Thus, a hypo-EP operator is a posinormal operartor with closed range.  Corollary 14 of \cite{JV} states that positive integral powers of hypo-EP operators are hypo-EP, which is equivalent to the Main Theorem above.  Because the Main Theorem is already in the literature, this paper will not be published. }
\smallskip

  \item  In \cite{Bouldin}, Bouldin characterizes when a product of two operators with closed range will have closed range and presents an example of an operator $T$ having closed range for which $T^2$ does not have closed range.  One key to proving the Main Theorem is establishing that if $T$ is posinormal with closed range, then $T^2$ also has closed range.  We present two different proofs, one of which is based on Bouldin's work. We also present an example simpler than Bouldin's of an operator $T$ with closed range such that $T^2$ does not have closed range.
  \end{itemize}

.. \section{Results}

\begin{lemma}\label{CRL}  Suppose that $T:\mathcal{H}\rightarrow \mathcal{H}$ is posinormal and has closed range; then $T^2$ also has closed range.  
\end{lemma}

\begin{proof}  Assume that $T$ has closed range, so that $T^*$ also has closed range. We have the following orthogonal decomposition of $\mathcal{H}$:  $\mathcal{H}= \ker T  \oplus \ran T^*$. Because the restriction of $T$ to $\ran T^*$ is one-to-one with closed range, $T$ is bounded below on $\ran T^*$:    for all $h\in \ran T^*$, we have
$$
\|Th\| \ge c \|h\|,\  \text{for some}\  c > 0.
$$
  Because $T$ is posinormal,  $\ran T\subseteq \ran T^*$ (by Part (2) of Theorem~\ref{posinormal}), and thus we have 
\begin{equation}\label{bddbelow}
\|T^2h\| \ge c^2 \|h\|  \quad \text{for all} \ h\in \ran T^*.
\end{equation}
It follows that $T^2$ has closed range: if $T^2 h_n \to w$ as $n\to\infty$ for some $w\in \mathcal{H}$ and sequence $(h_n)$ in $\ran T^*$; then, by (\ref{bddbelow}), because $(T^2 h_n)$ is Cauchy,  the sequence $(h_n)$ is also Cauchy, so that $(h_n)$ has limit $h\in \ran T^*$ and  $w = \lim_{n\to\infty} T^2 h_n  = T^2h$.   \end{proof}

  The preceding lemma may also be established using
      \begin{quotation}  Bouldin's Criterion \cite{Bouldin}: {\it  If $A$ and $B$ are operators on $\mathcal{H}$ having closed range then $AB$ also has closed range if and only if the angle between  $\ran B$ and $\ker A\cap(\ker A \cap \ran B)^\perp$ is positive.}
      \end{quotation}
Recall that the angle $\theta$ between two subspaces $\mathcal{M}$ and $\mathcal{N}$  of $\mathcal{H}$ is given by
$$
\theta = \cos^{-1}\left(\rule{0in}{.15in}\sup\{|\langle f, g\rangle|: f\in \mathcal{M}, g\in \mathcal{N}, \|f\| =1= \|g\|\}\right),
$$
where $\langle \cdot, \cdot\rangle$ denotes the inner product of the Hilbert space $\mathcal{H}$.  Suppose that $A = B = T$, where $T$ is posinormal with closed range.  Then $\ker T\cap \ran T =\{0\}$ because $\ran T\subseteq \ran T^*$.  Hence, $\ker T\cap (\ker T\cap \ran T)^\perp = \ker T$ and  the angle $\theta$ between $\ran T$ and $\ker T\cap(\ker T \cap \ran T)^\perp$ is
$$
\cos^{-1}\left(\rule{0in}{.15in}\sup\{|\langle f, g\rangle|: f\in \ran T, g\in \ker T, \|f\| =1= \|g\|\}\right) = \cos^{-1}(0) = \pi/2 > 0,
$$
where we have again used $\ran T \subseteq \ran T^*$ to obtain $\langle f, g\rangle = 0$ for all inner products in the set above whose supremum is found to be 0.  By Bouldin's Criterion, $T^2$ has closed range.  
     
  In general, if $T: \mathcal{H} \to \mathcal{H}$ has closed range, then $T^2$ need not have closed range.  Bouldin \cite[p.\ 363]{Bouldin} provides an example involving a direct-sum decomposition of $\mathcal{H}$ into three mutually orthogonal subspaces.  Here's an example involving a  direct-sum decomposition of $\mathcal{H}$ into two  orthogonal subspaces.
  
 \begin{exmp} A Hilbert-space operator having closed range whose square does not have closed range.  
 \end{exmp}
 \begin{proof}[]
  Let $A: \mathcal{H} \to  \mathcal{H}$  be any bounded linear operator on the Hilbert space $\mathcal{H}$ such that $A$ does not have closed range.   Define $T: \mathcal{H}\oplus \mathcal{H} \to \mathcal{H}\oplus \mathcal{H}$ by
 $$
 T(h_1, h_2) = (Ah_1 + h_2, 0), 
 $$
 so that $T$ has matrix representation 
 $$
 T  = \begin{pmatrix} A & I \\0&0\end{pmatrix}.
 $$
 Then $\ran T  = \mathcal{H}\oplus \{0\}$ is closed and $\ran T^2  = \ran A \oplus \{0\}$  is not closed. \end{proof}

 The following lemma is Proposition 3 of \cite{JKKP}.  We include a short proof for completeness.

\begin{lemma}[\cite{JKKP}]\label{keylemma}  Suppose that $T:\mathcal{H}\to \mathcal{H}$ is posinormal; then
$$
\ker T =  \ker T^2.
$$
\end{lemma}
\begin{proof}  That $\ker T \subseteq \ker T^2$ is clear.  Suppose that $h\in \ker T^2$.  Because $T$ is posinormal, $\ran T \subseteq \ran T^*$; thus, there is some $w\in \mathcal{H}$ such that $Th = T^*w$.  We have
$$
0 = T(Th) = T(T^*w),
$$
so that $0 = \langle TT^* w, w\rangle = \langle T^*w, T^* w\rangle = \|T^*w\|^2$.  Hence $T^*w = 0$ and we have $Th = T^*w = 0$. Thus, $h\in \ker T$, and it follows that $\ker T^2\subseteq \ker T$, as desired.
\end{proof}

   : 
       
   \begin{mthm} If $T: \mathcal{H}\to \mathcal{H}$ is a posinormal operator with closed range, then $T^n$ is posinormal for all $n\ge 1$ and $T^n$ has closed range for all $n\ge1$. 
\end{mthm}
\begin{proof}  Suppose that $T$ is posinormal with closed range.  By Lemma~\ref{CRL},  $T^2$ also has closed range.  Because  $\ran T$ and $\ran T^2$ are closed,  the same is true of $\ran T^*$ and  $\ran {T^*}^2$.   Because $T$ is posinormal, Lemma~\ref{keylemma} yields
$\ker T = \ker T^2$.    Thus,  $(\ker T)^{\perp} = (\ker T^2)^\perp$, so that  the closure of $\ran T^*$  equals the closure of $\ran {T^*}^2$, but these ranges are closed, so that $\ran T^* = \ran {T^*}^2$. Thus, $\ran T^* = \ran {T^*}^n$ for all $n\ge1$ (see, e.g., \cite[p.\ 290]{TL}).  Because $T$ is posinormal, $\ran T \subseteq \ran T^*$. Thus for every $n \ge 1$, we have
$$
\ran T^n \subseteq \ran T \subseteq \ran T^* = \ran {T^*}^n.
$$
Hence, by part (2) of Theorem~\ref{posinormal}, $T^n$ is posinormal for all $n\ge 1$.   

    Note that $T^n$ has closed range for all $n\ge1$ because $\ran T^*$ is closed and $\ran {T^*}^n = \ran T^*$ for all $n\ge 1$.
\end{proof}

Powers of a normal operator remain normal, of course, whether or not the operator has closed range.  However, powers of a hyponormal operator need not be hyponormal, and, as discussed in the Introduction, powers of posinormal operators need not be posinormal.   We have shown in our Main Theorem that adding to posinormality the hypothesis of closed range does yield the posinormality of powers.  Thus, we can say that a closed range hyponormal operator has posinormal powers, but  the stronger assertion that powers are hyponormal is not valid---we conclude this paper with an example illustrating that the square of a closed-range hyponormal operator need not be hyponormal.    First, we state the following immediate corollary of our Main Theorem. 

\begin{cor} If $T: \mathcal{H}\to \mathcal{H}$ has closed range and is normal or hyponormal , then $T^n$ has closed range for all $n\ge 1$.
\end{cor}
That powers of closed-range normal operators continue to have closed range is a consequence of the Spectral Mapping Theorem  and the following characterization of normal operators having closed range (see, e.g., Proposition 4.5 of Chapter 11 of \cite{JBC}): 
\begin{quotation} {\it A normal operator has closed range if and only if $0$ is not a limit point of its spectrum.}
\end{quotation}
The preceding spectral characterization of close-range normal operators certainly does not extend to hyponormal  operators; for example, the unilateral shift $U$ on $\ell^2$ is hyponormal with closed range but its spectrum is the closed unit disk.  That powers of closed-range hyponormal operators continue to have closed range appears to be new. 

\begin{exmp} A hyponormal operator having closed range whose square is not hyponormal.  
 \end{exmp}
 \begin{proof}[]
  Let $U$, as above, be the unilateral shift on $\ell^2$. Ito and Wong \cite[p.\ 158]{IW}  point out  $U^* + 2U$ is an example of a hyponormal Toeplitz operator ($U^* + 2U$ is equivalent to the Toeplitz operator with symbol  $\phi(z) = \bar{z} + 2 z$) that is neither normal nor analytic.   The hyponormal operator $U^* + 2U$ has closed range because it is bounded below: for each $\ell^2$ sequence $s$, we have $\|Us\| = \|s\|$ and $\|U^*s\| \le \|s\|$ so that  $\|(2U + U^*)s\| \ge \|2Us\| - \|U^*s\| \ge \|s\|$.  However, the square of $U^* + 2U$ is not hyponormal---see the solution of Problem 209 of \cite{PH82}. (That all powers of $U^* + 2U$ remain posinormal follows from our Main Theorem or the observation that all powers of the adjoint of $U^* + 2U$ are surjective because $U^* + 2U$ is bounded below.)    \end{proof}


\end{document}